\documentclass[12pt,a4paper]{amsart}

\title[Malnormal subgroups]{Malnormal subgroups of lattices and the Puk\'anszky invariant in group factors}
\date{July 17, 2009}
\subjclass{22D25, 22E40, 20G20}
\author{Guyan Robertson }
\address{School of Mathematics and Statistics, Newcastle University, Newcastle upon Tyne NE1 7RU, England, U.K.}
\email{a.g.robertson@newcastle.ac.uk}
\author{Tim Steger}
\address{Struttura di Matematica e Fisica, Universit\`a di
Sassari, Via Vienna 2, 07100 Sassari, Italia}
\email{steger@ssmain.uniss.it}

\usepackage{amsmath}
\usepackage{amscd}

\input{prepictex}
\input{pictex}
\input{postpictex}
\chardef\bslash=`\\ 

\makeatletter
\def\verbatim{\interlinepenalty\@M \@verbatim
  \leftskip\@totalleftmargin\advance\leftskip2pc
  \frenchspacing\@vobeyspaces \@xverbatim}
\makeatother
\newtheorem{theorem}{Theorem}[section]
\newtheorem{corollary}[theorem]{Corollary}

\newtheorem{proposition}[theorem]{Proposition}

\theoremstyle{definition}

\newtheorem{remark}[theorem]{Remark}

\newcommand{\bb}[1]{{\mathbb{#1}}}
\newcommand{\fk}[1]{{\mathfrak{#1}}}

\newcounter{picture}

\newcommand{\tto}{{\text{\rm{II}}}_{1}}

\newcommand{\vn}{{\text{\rm VN}}}

\begin{document}

\begin{abstract}
Let $G$ be a connected semisimple real algebraic group. Assume that $G(\bb R)$ has no compact factors and let $\Gamma$ be a torsion-free uniform lattice subgroup of $G(\bb R)$.  Then $\Gamma$ contains a malnormal abelian subgroup $A$. This implies that the
$\tto$ factor $\vn(\Gamma)$ contains a masa $\fk A$ with Puk\'anszky invariant $\{\infty\}$.
\end{abstract}

\maketitle

\section{Introduction}

A subgroup $\Gamma_0$ of a group $\Gamma$ is {\it malnormal}
if $x\Gamma_0x^{-1} \cap \Gamma_0 = \{ 1 \}$ for all
$x \in \Gamma - \Gamma_0$. An abelian malnormal subgroup is necessarily maximal abelian.
The main result of this article is Theorem \ref{main}, which rests upon work of Prasad and Rapinchuk \cite{PR}.

\begin{theorem}\label{main} Let $G$ be a connected semisimple real algebraic group  and let $d$ be the $\bb R$-rank of $G$. Assume that $G(\bb R)$ has no compact factors and let $\Gamma$ be a torsion-free uniform lattice subgroup of $G(\bb R)$.  Then $\Gamma$ contains a malnormal abelian subgroup $A\cong\bb Z^d$.
\end{theorem}

Theorem \ref{main} will be applied to the group factor $\vn(\Gamma)$.
Recall that if $\Gamma$ is a group, then the von Neumann algebra $\vn(\Gamma)$ is the convolution algebra
$$\vn(\Gamma)=\{f\in\ell^2(\Gamma)\, :\, f\star \ell^2(\Gamma)\subseteq\ell^2(\Gamma)\}\,.$$
It is well known that if $\Gamma$ is an infinite conjugacy class group then $\vn(\Gamma)$ is a factor of type $\tto$.
This is true if $\Gamma$ is a lattice in a semisimple Lie group \cite[Lemma 3.3.1]{GHJ}.
If $\Gamma_0$ is a subgroup of $\Gamma$, then
$\vn(\Gamma_0)$ embeds naturally as a subalgebra of $\vn(\Gamma)$
 via $f\mapsto \overline f$, where
\begin{equation*} \overline f(x) =\begin{cases}
f(x)& \text{if $x\in\Gamma_0$},\\
0& \text{otherwise}.
\end{cases}
\end{equation*}
This article is concerned with examples where $\Gamma_0=A$ is an abelian
subgroup of $\Gamma$ and $\fk A = \vn(A)$ is a maximal abelian
$\star$-subalgebra (masa) of $\fk M = \vn(\Gamma)$. Recall that
$\fk A$ is the von Neumann subalgebra of $B(\ell^2(\Gamma))$
defined by the left convolution operators
$$\lambda(f) : \phi\mapsto f\star \phi$$
where $f\in\ell^2(A)$ and $f\star
\ell^2(\Gamma)\subseteq\ell^2(\Gamma)$. The algebra $\fk A$ also
acts on $\ell^2(\Gamma)$ by right convolution
$$\rho(f) : \phi\mapsto \phi\star f.$$
Let $\fk A^{opp}$ be the von Neumann
subalgebra of $B(\ell^2(\Gamma))$ defined by this right action of
$\fk A$.
Let ${\fk B}$ be the von Neumann subalgebra of $B(\ell^2(\Gamma))$ generated by $\fk A~\cup~\fk A^{opp}$
 and let $p$ denote the orthogonal projection of
$\ell^2(\Gamma)$ onto the closed subspace generated by $\fk A$.
Then $p$ is in the centre of the commutant ${\fk B}'$, and ${\fk B}'p$ is
abelian. The von Neumann algebra ${\fk B}'(1 - p)$ is of type $I$
and may therefore be expressed as a direct sum $\fk B_{n_1}\oplus
\fk B_{n_2} \oplus \dots$ of algebras $\fk B_{n_i}$ of type
$I_{n_i}$, where $1\le n_1 < n_2 < \dots \le \infty$. The {\it
Puk\'anszky invariant} \cite [Chapter 7]{SS2} is the set $\{n_1, n_2, \dots
\}$.  It is an isomorphism invariant of the pair $(\fk A,\fk M)$. It has been shown \cite[Corollary 3.3]{NS}
that each nonempty subset $S$ of the natural numbers containing $1$ can be realized as the
Puk\'anszky invariant of some masa in the hyperfinite $II_1$
factor $R$. This was extended \cite{SS, DSS} to subsets $S$ containing $\infty$  for $R$ and for the free group factor.
It was later extended \cite{Wh} to arbitrary subsets $S~\subset~\bb N~\cup\{\infty\}$ for $R$ (and for certain other McDuff factors).

It is known that every factor of type $\tto$ contains a singular masa \cite{Pop}. S. Popa \cite [Remark 3.4]{po} showed that if the Puk\'anszky invariant of $(\fk A,\fk M)$ does not contain $1$, then $\fk A$ is a singular masa in $\fk M$. K. Dykema \cite{Dyk} has  shown (using Voiculescu's free entropy dimension) that the Puk\'anszky invariant of any masa in the free group factor must either contain $\infty$ or be unbounded.
This means that it is not possible for any singleton other than $\{\infty\}$ to be a possible Puk\'anszky invariant occurring in every $\tto$ factor.
Jolissaint \cite{Jol} has shown that
if $F_0$ is the cyclic subgroup generated by the first generator of Thompson's group $F$ then $\vn(F_0)$ has Pukansky invariant $\{\infty\}$.
A natural question arises.
\begin{itemize}
\item[$\bullet$] Does every $\tto$ factor $\fk M$ contain a masa $\fk A$ with Puk\'anszky invariant $\{\infty\}$?
\end{itemize}

This article uses Theorem \ref{main} to provide an affirmative answer for $\fk M = \vn(\Gamma)$, where $\Gamma$ is a torsion-free uniform lattice subgroup of a connected semisimple real algebraic group $G$ without compact factors. If $G$ has $\bb R$-rank $\ge 2$, then $\vn(\Gamma)$ has Kazhdan's property $T$. This is the first result on possible values of the Puk\'anszky invariant in a $\tto$ factor with property $T$.

Many thanks are due to G. Prasad and Y. Shalom for their assistance with the proofs of Theorem \ref{main} and Theorem \ref{double} respectively. S. White provided valuable background information.

\bigskip

\section{Malnormal abelian subgroups of lattices}

This section is devoted to the proof of Theorem \ref{main}.
Let $G$ be a connected semisimple real algebraic group  and let $d$ be the $\bb R$-rank of $G$. Assume that $G(\bb R)$ has no compact factors
and let $\Gamma$ be a torsion-free uniform lattice subgroup of $G(\bb R)$.

Since $\Gamma$ is finitely generated, $\Gamma < G(\bb K)$ for some finitely generated subfield $\bb K$ of $\bb R$. By the Borel Density Theorem
\cite[Chapter II, Corollary 4.4]{Mar}, $\Gamma$ is Zariski dense in $G(\bb R)$. Therefore, according to Theorems 1 and 2 of \cite{PR}, there exists a maximal abelian torus subgroup $T$ of $G$ with the following properties.
\begin{itemize}
\item[(1)] The $\bb R$-rank of $T$ is $d$.
\item[(2)] $A = T(\bb R)\cap\Gamma$ is a uniform lattice in $T(\bb R)$.
\item[(3)] $T$ has no proper algebraic subgroups defined over $\bb K$.
\end{itemize}
Moreover, $T$ is the $\bb K$-Zariski closure of a single $\bb R$-regular element $x_0$ in $A$ \cite{PR}.
In fact there are many such elements $x_0$ \cite[Remark 2]{PR}.
We claim that $A$ is a malnormal subgroup of $\Gamma$.
To this end, fix an arbitrary element $x \in \Gamma - A$. We must show that $xAx^{-1}~\cap A~=~\{ 1 \}$.

Let $T_s$ (respectively $T_a$) be the maximal $\bb R$-split (respectively $\bb R$-anisotropic) subtorus of $T$. Then $T=T_s\cdot T_a$ (an almost direct product) \cite [Proposition 8.15]{Bo} and $T_s$ is a maximal $\bb R$-split torus in $G$ \cite[Remark 1]{PR}. Thus $T(\bb R) = S \cdot C$,  where $S=T_s(\bb R)\cong \bb R^d$ and $C=T_a(\bb R)\cong (\bb R/\bb Z)^r$, where $r$ is the dimension of $T_a$.
Since $\Gamma$ is torsion free and discrete, $A$ is a uniform lattice in $S$. In particular, $A\cong\bb Z^d$.

Since $T$ is the $\bb K$-Zariski closure of $A$, it follows that $T$ is defined over
$\bb K$. Since $x\in\Gamma\subseteq G(\bb K)$,
$x T x^{-1}$ is also defined over $\bb K$. According to condition (3), there are only two possibilities:

\begin{itemize}
\item[$\bullet$] $T \cap x T x^{-1}=\{1\}$;
\item[$\bullet$] $T = x T x^{-1}$.
\end{itemize}

In the first case we also have $T(\bb K) \cap x T(\bb K) x^{-1}=\{1\}$, and {\it a fortiori} $A \cap x A x^{-1}=\{1\}$, as required.

To show that the second case does not occur,
assume that $T~=~x T x^{-1}$. This implies that $T(\bb R)$ is stable under
conjugation by $x$.  Also $\Gamma$ is stable under
conjugation by $x$.
Therefore $x A x^{-1} = A$, since $A = \Gamma\cap T(\bb R)$.
There are two possibilities to consider for the action $\alpha_x : a \mapsto xax^{-1}$
on $A$.

\begin{itemize}
\item[(a)] $\alpha_x$ fixes only the trivial element of $A$;
\item[(b)] $\alpha_x$ fixes some nontrivial element of $A$.
\end{itemize}

Since conjugation by $x$ stabilizes $T$, it also stabilizes
$T_s$ and $T_a$ separately \cite[Proposition 8.15(3)]{Bo}. Thus $xSx^{-1}=S$.
The symmetric space of $G(\bb R)$ is $X=G(\bb R)/K$, where $K$ is a maximal compact subgroup of $G(\bb R)$.
The group $\Gamma$  acts freely on $X$, since it is torsion free.
Since $T_s$ is a maximal $\bb R$-split torus of $G$, there is a unique flat $F$ in $X$ such that
$SF=F$ and $S$ acts simply transitively on $F$ \cite[Lemma 5.1]{Mos}. Now $xF$ is another such flat, since
$$SxF = (xSx^{-1})xF = xSF = xF.$$
Hence $xF = F$.

The action of $x$ on $F$ is by some rigid motion and the action of $A$ on $F$ is by translations.
No nontrivial element in $\Gamma$ can act trivially on $F$, so we can
calculate the conjugation by $x$ of any element $y\in A$ by
considering the actions of these elements on $F$.
The two cases above correspond to:

\begin{itemize}
\item[(a)] $x$ acts on $F$ by a rigid motion whose linear part has trivial
  $1$-eigenspace;
\item[(b)]  $x$ acts on $F$ by a rigid motion whose linear part has
  nontrivial $1$-eigenspace.
\end{itemize}

In case (a), $x$ necessarily has a fixed point in $F$.
Therefore $x=1$, since $\Gamma$ acts freely on $X$.

Consider case (b). The algebraic subgroup $Z_G(x)\cap T$, which
consists of the elements commuting with $x$, is defined over $\bb
K$.  In case (b), $Z_G(x)\cap T$ contains nontrivial elements of $A$.
That is, it has nontrivial $\bb K$-points.  Hence, it must be nontrivial
(as an algebraic group). By condition (3) it must be all of $T$.  Hence $x$
commutes with every element of $T$.  Therefore the algebraic closure of
$\{x\} \cup T$ over $\bb K$ is commutative.  However $T$ is a maximal
abelian subgroup over $\bb K$, and so $x\in T(\bb
K)$. Therefore $x\in A$, contrary to assumption.
This completes the proof of Theorem \ref{main}.

\bigskip
\section{The Puk\'anszky invariant}

The following result was proved in \cite[Proposition 3.6]{RS} and later extended in \cite[Theorem 4.1]{SS}.
\begin{proposition} \label{3e} Suppose that $A$ is an abelian subgroup of a countable group $\Gamma$
such that  $\fk A = \vn(A)$ is a masa of $\vn(\Gamma)$. If $A$ is
malnormal in $\Gamma$ then the Puk\'anszky invariant of $\fk A$ contains precisely one element $n = \#(A
\backslash \Gamma /A-\{A\})$.
\end{proposition}

In view of Theorem \ref{main}, the next result is enough to provide examples of masas with Puk\'anszky invariant $\{\infty\}$.

\begin{theorem}\label{double}
Let $G$ be a connected semisimple real algebraic group. Assume that $G(\bb R)$ has no compact factors.
Let $\Gamma$ be a torsion free uniform lattice subgroup of $G(\bb R)$ and let $A < \Gamma$ be an abelian subgroup.Then
     $$\#(A \backslash \Gamma / A) = \infty.$$
\end{theorem}

\begin{proof} Suppose that $\#(A \backslash \Gamma / A) < \infty$.
Then
$$\Gamma = \bigcup_{x\in F} A x A$$
 where $F\subset \Gamma$ is finite.  Taking Zariski closures,
 it follows from the Borel Density Theorem that
\begin{equation}\label{one}
  G(\bb R) = \bar\Gamma = \bigcup_{x\in F} \overline{A x A}.
\end{equation}
For each $x\in F$, $\bar A x \bar A$ is locally closed in the Zariski topology, since it is an orbit of $\bar A$ acting on $G(\bb R)/\bar A$ \cite[Corollary 3.1.5]{Z}.  This means that $\bar A x \bar A = U\cap E$ where $U$ is Zariski-open and $E$ is Zariski closed.  Therefore
$$\overline{A x A}
  = \overline{U\cap E}
  \subseteq  E
  \subseteq  (U\cap E) \cup (G(\bb R)\setminus U)
  = (\bar A x \bar A) \cup (G(\bb R)\setminus U).$$

  Since $U\subseteq G(\bb R)$ is Zariski open (and $G(\bb R)$ is Zariski connected), $G(\bb R)\backslash U$ has measure zero, relative to Haar measure $\mu$ on $G(\bb R)$ \cite[Chapter I, Proposition 2.5.3]{Mar}. Therefore,
  \begin{equation}\label{two}
  \mu(\overline {A x A})=\mu(\bar A x \bar A).
  \end{equation}
  Now we show that $\mu(\bar A x \bar A)=0$.
Each element of $A$ is semisimple, since each element of $\Gamma$ is \cite[Section 11]{Mos}. Therefore each element of the Zariski closure $\bar A$ is also semisimple; in other words, $\bar A$ is a torus subgroup. The dimension of $\bar A$ as a Lie group is no larger than the absolute rank $d'$ of $G$ (the rank over $\bb C$ of the Lie algebra of $G$).
The number of positive roots of the compexified Lie algebra of $G$ is at least $d'$, and the total number of roots is at least $2d'$. Thus the total dimension of the root spaces is at least $2d'$.  This means that $d'$ is at most one third of the dimension of $G$. Thus the dimension of $\bar A\times\bar A$ is at most two thirds the dimension of $G$.
The map $(a_1,a_2)\mapsto a_1 x a_2$ from $\bar A\times\bar A$ to $G(\bb R)$ is $C^\infty$ (in fact polynomial).  Therefore, by the above dimension count, its image has measure zero.  It follows from (\ref{two}) that $\mu(\bar A x \bar A)=0$. However, this contradicts (\ref{one}), thereby proving the result.
\end{proof}

An immediate consequence of Theorem \ref{main} and Theorem \ref{double} is

\begin{corollary} Let $G$ be a connected semisimple real algebraic group such that $G(\bb R)$ has no compact factors. Let $\Gamma$ be a uniform lattice subgroup of $G(\bb R)$. Then there exists an abelian subgroup $A<\Gamma$ such that $\vn(A)$ is a masa of $\vn(\Gamma)$ with Puk\'anszky invariant $\{\infty\}$.
\end{corollary}

\begin{proof}
  This follows immediately from Proposition \ref{3e}.
\end{proof}

\begin{remark}
  A similar result was obtained by geometrical methods in \cite[Theorem 4.6]{Ro}, if $\Gamma$ is the fundamental group of a compact locally symmetric space $M$ of constant negative curvature and $A$ is generated by the homotopy class of a simple closed geodesic in $M$.
\end{remark}

\end{document}